\documentclass[a4paper,12pt,onecolumn,twoside]{article}
\usepackage{hyperref}
\usepackage[numbers,sort&compress]{natbib}

\date{}
\makeatletter
\@addtoreset{equation}{section}
\makeatother
\usepackage{amssymb,amsmath}
\usepackage{amsthm}
\allowdisplaybreaks

\begin{document}

\title{\vspace{-38pt}
\bf Three solutions for a new Kirchhoff-type problem\footnote{Supported by the NSF of China ( 11661021,11861021).}}
\author{
{Yue Wang$^1$
, Hong-Min Suo$^2$\footnote{Corresponding author. E-mail:eztf@qq.com.
}
}\\[8pt] E-mail: yuewn@sina.cn\\
{\small \emph{1. School of Mathematics and Statistics,
Guizhou University, }}\\
{\small \emph{Guiyang {\rm 550025}, People's Republic of China}}\\
{\small \emph{2. School of Data Science and Information Engineering,
Guizhou Minzu University, }}\\
{\small \emph{Guiyang {\rm 550025}, People's Republic of China}}
}

\maketitle
\numberwithin{equation}{section}
\newtheorem{theorem}{Theorem}[section]
\newtheorem{lemma}[theorem]{Lemma}
\newtheorem{remark}[theorem]{Remark}
\newtheorem{corollary}[theorem]{Corollary}

\baselineskip 16pt
\vspace{-27pt}
\begin{center}
\begin{minipage}[c]{15.4cm}
\hspace{0em}\textbf{Abstract:}
This article concerns on the existence of multiple solutions for a new Kirchhoff-type problem with negative modulus. We prove that there exist three nontrivial solutions when the parameter is enough small via the variational methods and algebraic analysis. Moreover, our fundamental technique is one of the Mountain Pass Lemma, Ekeland variational principle, and Minimax principle. \\[-8pt]

\hspace{0em}\textbf{Keywords:}
New Kirchhoff-type problem; variational method; algebraic analysis
\\[-8pt]

\hspace{0em}\textbf{MSC(2010):}
 35A15, 35B33, 35J62
\end{minipage}
\end{center}

\section{Introduction and main results}

This paper research the following  nonlocal problem\\[-5pt]
\begin{equation}\label{wt1}
 \left\{
\begin{array}{lc}\displaystyle
- \Big(a- b \int_\Omega|\nabla u|^2dx \Big)\Delta u = \mu f(x),  &\text{ in }  \Omega,\\
u=0,  &\text{ on }  \partial\Omega,
\end{array}\right.
\end{equation}
where $a, b>0$ are constants, $\mu\in\mathbb{R}$ is a parameter, $f(x)$ is a nonnegative-nonzero function,
and  $\Omega$ is a bounded domain in $\mathbb{R}^N(N\geq1)$ with smooth boundary $\partial\Omega$.
The analysis developed in this paper
corresponds to propose an approach based on the idea of considering the nonlocal term with negative modulus which
is presents interesting difficulties.

\vspace{4pt}
Problem \eqref{wt1} is related with the Kirchhoff-type equation as following:\\[-6pt]
\begin{equation}\label{K1}
\displaystyle
-\Big(a+b\int_\Omega|\nabla u|^2dx\Big)\Delta u=f(x,u), \text{ in }  \Omega,
\end{equation}
where $a\geq0$, $b>0$, $\Omega$ is a bounded domain in $\mathbb{R}^N$ with smooth boundary $\partial\Omega$
or $\Omega=\mathbb{R}^N$, $f: \Omega\times\mathbb{R}\to\mathbb{R}$  is a continual function.
Mentioned that Eq.\eqref{K1} is a subproblem of model\\[-14pt]
\begin{align}\label{wt0}
\varrho h \frac{\partial^2u}{\partial t^2}
-\Big( p_0+\frac{Eh}{2L}\int_0^L\Big|\frac{\partial u}{\partial x}\Big|^2dx \Big)\frac{\partial^2 u}{\partial x^2}= f(x,u)
\end{align}
with $0<x<L$ and $t\geq0$, where $u=u(x,t)$ is the lateral displacement,
$\varrho$ the mass density, $E$ the Young's modulus,
$h$ the cross-section area, $L$ the length, $p_0$ the initial axial tension.
Eq.\eqref{wt0} named the Kirchhoff problem as an extension of classical D'Alembert's wave equations for free vibration of elastic strings by Kirchhoff \cite{Kirchhoff} in 1876. When finding the existence of stationary solution, Eq.\eqref{wt0} may be express as the Eq.\eqref{K1} and therefore problem \eqref{K1} were named the Kirchhoff-type problem.
Eq.\eqref{K1} has been studied by many researchers on whole space $\mathbb{R}^N$ and bounded domain with some boundary conditions, such as \cite{ChenST,LiFY,Faraci,GuoZJ,Bouizem,SunYJ,LuSS,Gongbao Li1,ChenST2,CY Lei0,Tang XH} and their references.
Problem \eqref{K1} contains a nonlocal coefficient $(a+b\int_\Omega|\nabla u|^2dx)$, this leads to that Eq.\eqref{K1} is no longer a pointwise identity and therefore it is often called nonlocal problem.
It is worth paying more attentions to Young's modulus, which can also be used in computing tension, when the atoms are pulled apart instead of squeezed together, the strain is negative because the atoms are stretched instead of compressed, this leads to minus Young's modulus. Therefore, a nonlocal coefficient $(a-b\int_\Omega|\nabla u|^2dx)$ is included in Kirchhoff-type problem, may be an interesting model(see  \cite{WYEJDE}), which is the reason why we call it Kirchhoff-type problem with negative modulus in the abstract.

\vspace{5pt}
Indeed, let $a,b>0$, Yin and Liu in \cite{GS Yin} got that there exist at least a nontrivial non-negative solution and a nontrivial non-positive solution
for the following problem\\[-8pt]
\begin{equation*}
\left\{
\begin{array}{lc}\displaystyle
- \Big(a-b \int_{\Omega} |\nabla u|^2 dx \Big)\Delta u=|u|^{p-2}u , &\text{ in }  \Omega,\\
u=0,   &\text{ on }  \partial\Omega,
\end{array}
\right.
\end{equation*}
where $\Omega$ is a bounded domain in $\mathbb{R}^N$ with smooth boundary,
$2<p<2^*$, $2^*=\frac{2N}{N-2}$ as $N\geq3$ and $2^*=+\infty$ as $N=1,2$.
In addition, infinitely many solutions were got in \cite{WYAA} included the condition above and $p=2$ via the Ljusternik-Schnirelman type minimax method.
Wang et al. \cite{WYEJDE} obtained  that
there exist at least two positive solutions when $f(x)\in L^{4/3}(\mathbb{R}^4)$ with $\mu>0$ small and infinitely many positive solutions with $\mu=0$ via variational method mainly for the problem\\[-8pt]
\begin{equation*}\displaystyle
 -\Big(a-b\int_{\mathbb{R}^4}|\nabla u|^2dx\Big)\Delta u=|u|^2u+\mu f(x), \quad \text{in } \mathbb{R}^4.
\end{equation*}
The research interesting in \cite{GS Yin} is that the nonlocal coefficient $(a+b\int_\Omega|\nabla u|^2dx)$ is bounded below but the nonlocal coefficient $(a-b\int_\Omega|\nabla u|^2dx)$ is not. Different with \cite{GS Yin},
the research interesting in \cite{WYEJDE} is that the Kirchhoff-type equation with nonlocal coefficient $(a-b\int_\Omega|\nabla u|^2dx)$ is a negative modulus problem.
For more details about the Kirchhoff problems with negative modulus, we refer readers to
\cite{WYYYFH,CY Lei1,CY Lei2,DYJNSA,ZJADE,LJPYYSX,
WYYYSX,WYYYFH2,HamdaniNA,WYSXWL,QXT3,WYXDZK,WYGXSD,CZPCQSD,WYCDLG}.
Enlightened by the papers above, we consider the problem \eqref{wt1}, the conclusion state as following theories main via variational method and algebraic analysis.\\[-14pt]

\begin{theorem} \label{thm1.1}
Assume that $a, b>0$ and $f(x)\in L^\frac{2^*}{2^*-1}(\Omega)$ is positive a.e., then, there exists $\mu_*>0$, such that the problem \eqref{wt1} has at least three nontrivial solutions for  $\mu\in(0,\mu_*)$  and a nontrivial solution for $\mu\in[\mu_*,+\infty)$.
\end{theorem}

\vspace{-3pt}
As the proof as Theorem \ref{thm1.1}, the proof of the existence of solutions for $\mu<0$ may be miscellaneous by variational method, we shall use new method to overcome the difficulties.\\[-20pt]
\begin{theorem} \label{thm1.2}
Assume that $a, b>0$ and $f(x)\in L^\frac{2^*}{2^*-1}(\Omega)$ is positive a.e., then, there exists a constant  $\mu_{**}>0$
such that problem \eqref{wt1} has only three solutions for $0<|\mu|<\mu_{**}$, only two solutions for $\mu=\pm\mu_{**}$ and unique solution for $|\mu|>\mu_{**}$. Moreover, problem \eqref{wt1} has infinitely many solutions for $a,b>0,\mu=0$.
\end{theorem}
Indeed, the condition of $f(x)$ can be replace weakly by $f\in L^2(\Omega)$ and $f(x)>0$ a.e..
\begin{corollary} \label{thm1.3}
Assume that $a, b>0$ and $f(x)\in L^2(\Omega)$ is positive a.e., then, for $\mu_{**}$ defined by Theorem \ref{thm1.2}, problem \eqref{wt1} has only three solutions for any $0<|\mu|<\mu_{**}$, only two solutions for $\mu=\pm\mu_{**}$ and unique solution for any $|\mu|>\mu_{**}$.
\end{corollary}

\vspace{-3pt}
The novelty of our results lies in three aspects. Firstly, different with those articles mentioned above,
we prove the $(PS)_c$ condition by new method and $c\in[-\frac{a^2}{12b},+\infty)\setminus\{\frac{a^2}{4b}\}$,
what's more, it is not satisfy $(PS)_c$ condition with $c\in(-\infty,-\frac{a^2}{12b})\cup\{\frac{a^2}{4b}\}$.
Secondly, with the help of algebraic analysis, we got that there are at least three nontrivial weak solutions by using one of the Mountain Pass Lemma, Ekeland variational principle, and Minimax principle.
Thirdly, through of a basic fundamental result, we obtain the specific form of $\mu_{**}$.

\vspace{4pt}
This article is organized as follows. In section 2, we give some
basic knowledge which use to solving the problem. Section 3 contains elementary results and proof of theorem \ref{thm1.1}. In section 4, for the Theorem \ref{thm1.1}, with the help of algebraic analysis, we prove that the existence of three nontrivial weak solutions only by using one of the Mountain Pass Lemma, Ekeland variational principle, and Minimax principle. By using the similar method, we prove main results of the Theorem \ref{thm1.2} and calculate the $\mu_{**}$ exactly.

\section{Preliminaries}

Throughout this paper
we denote by $\to $ (resp. $\rightharpoonup$) the strong (resp. weak) convergence.
For any $u, v\in H_0^1(\Omega)$, the inner product is
$\langle u,v\rangle=\int_\Omega\nabla u \nabla vdx$ and the norm
$\|u\|= \big(\int_\Omega|\nabla u|^2dx \big)^\frac{1}{2}.
$

\vspace{5pt}
We recall that a function $u\in H_0^1(\Omega)$ is called a weak solution of problem \eqref{wt1} if
$$
\Big(a-b\int_\Omega|\nabla u|^2dx \Big)\int_\Omega\nabla u \nabla vdx
=\mu\int_{\Omega}f vdx,
~~\forall~v\in H_0^1(\Omega).$$
Denote the  $L^s$-norm $\|\cdot\|_s=\big[\int_\Omega|u|^sdx\big]^\frac{1}{s}$ for $0< s<+\infty$ and
$\displaystyle \|\cdot\|_\infty= ess \sup_{x\in\Omega} |u|$ ( still denoted by $\|\cdot\|_s=\big[\int_{\Omega}|u|^sdx\big]^\frac{1}{s}$ ) for $s=+\infty$.
Set
$$
S= \inf_{u \in H_0^1(\Omega)\backslash\{0\}}
\frac{ \int_\Omega|\nabla u|^2dx}{ \big(\int_\Omega|u|^{2^*}dx\big)^\frac{2}{2^*}}, ~~
\lambda_1= \inf_{u \in H_0^1(\Omega)\backslash\{0\}}
\frac{ \int_\Omega|\nabla u|^2dx}{ \int_\Omega|u|^2dx}.
$$
Let $I(u): H_0^1(\Omega)\mapsto \mathbb{R}$ be the functional defined by
\begin{equation}\label{e-fun}
I(u)=\frac{a}{2}\|u\|^2-\frac{b}{4}\|u\|^4 -\mu\int_{\Omega}fudx,
\end{equation}
it can be verify that $I(u)\in C( H_0^1(\Omega),\mathbb{R})$ and the G\^{a}teaux derivative of $I$ given by
\begin{equation}\label{d-e-fun}
\langle I'(u), v\rangle= \displaystyle (a-b\|u\|^2)\int_\Omega\nabla u \nabla vdx  -\mu\int_\Omega fvdx,
~~\forall~v\in  H_0^1(\Omega).
\end{equation}
If $u\in  H_0^1(\Omega)$ such that $I'(u)=0$, then $u$ is a weak solution of problem \eqref{wt1}.

\section{Proof of  the Theorem \ref{thm1.1}}

\begin{lemma} \label{lem.PSc}
Assume that $a, b, \mu >0$, $f \in L^\frac{2^*}{2^*-1}(\Omega)$ and $f(x)\geq0$ a.e., then,
$I$ satisfies the $(PS)_c$ condition with $c\in[-\frac{a^2}{12b},+\infty)\setminus\{\frac{a^2}{4b}\}$, and
$I$ does not satisfy the $(PS)_c$ condition at $c=\frac{a^2}{4b}$.
That is, for $c\in[-\frac{a^2}{12b},+\infty)\setminus\{\frac{a^2}{4b}\}$, every $(PS)$ sequence at $c$ has a converge subsequence.
\end{lemma}
\begin{proof}[\bf Proof]
Recalling that $\{u_n\}\subset H_0^1(\Omega)$ is a $(PS)$ sequence at $c$, so that $I(u_n) \to c$ and $I'(u_n) \to0$ as $n\to\infty$.
Let $\{u_n\}$ be a $(PS)_c$ sequence, then,
for all $v\in H_0^1(\Omega)$, there holds
\begin{equation}\label{qrsl}
\begin{cases}
\displaystyle
I(u_n)=\frac{a}{2}\|u_n\|^2-\frac{b}{4} \|u_n\|^4 -\mu\int_{\Omega}fu_ndx \to c,\\
\langle I'(u_n), v\rangle= \displaystyle \bigr(a -b\|u_n\|^2\bigr)\int_{\Omega}\nabla u_n\nabla vdx -\mu\int_{\Omega}fvdx \to 0
\end{cases}
\end{equation}
as $n\to\infty$. Especially, taking $v=u_n$ in \eqref{qrsl}, it holds that
\begin{align}\label{shouliana24b}
I(u_n)-\langle I'(u_n), u_n\rangle=
\frac{3b}{4} \|u_n\|^4-\frac{a}{2}\|u_n\|^2 \to c.
\end{align}
Hence we can obtain that, there is
$$
\lim_{n\to\infty}\|u_n\|^2=\frac{\frac{a}{2}\pm\sqrt{\frac{a^2}{4}+3bc}}{2\times \frac{3b}{4}}
=\frac{a\pm a\sqrt{1+\frac{12bc}{a^2}}}{3b}=\frac{a}{3b}\Bigm( 1\pm\sqrt{1+\frac{12bc}{a^2}} \Bigm).
$$
From the above, we have  \\[-8pt]
\begin{equation}\label{fenlei-c}
\left\{\begin{array}{rccr}
\displaystyle \lim_{n\to\infty}\|u_n\|^2<0,   &&\text{if}&\displaystyle c<-\frac{a^2}{12b};\\[6pt]
\displaystyle 0\leq\lim_{n\to\infty}\|u_n\|^2\leq\frac{2a}{3b},  &&\text{if}&\displaystyle -\frac{a^2}{12b}\leq c\leq0 ;\\[6pt]
\displaystyle \frac{2a}{3b}<\lim_{n\to\infty}\|u_n\|^2<\frac{a}{b},  &&\text{if}&\displaystyle 0< c<\frac{a^2}{4b} ;\\[6pt]
\displaystyle \lim_{n\to\infty}\|u_n\|^2=\frac{a}{b}, &&\text{if}&\displaystyle c=\frac{a^2}{4b} ;\\[6pt]
\displaystyle \lim_{n\to\infty}\|u_n\|^2>\frac{a}{b}, &&\text{if}&\displaystyle c>\frac{a^2}{4b}.
\end{array}\right.
\end{equation}
Therefore, $c\geq-\frac{a^2}{12b}$ is well defined from \eqref{fenlei-c}, and there exists no $\displaystyle \lim_{n\to\infty}\|u_n\|^2$ for $c<-\frac{a^2}{12b}$.

Case ${\displaystyle \lim_{n\to\infty}\|u_n\|^2}=\frac{a}{b}$,
we have $a-b\|u_n\|^2\to0$ as  $n\to\infty$.
By (3,1), we can choose a  $v\in H_0^1(\Omega)$ such that $\int_{\Omega}fv dx=1$,  it holds that $\|v\|<+\infty$ and\\[-5pt]
$$
 \Bigl(a-b\|u_n\|^2 \Bigl)\int_{\Omega}\nabla u_n\nabla v dx\to \mu,
$$
this leads to that\\[-5pt]
$$
\lim_{n\to\infty}\|u_n\|\|v\|\geq
\lim_{n\to\infty}\int_{\Omega}\nabla u_n\nabla v dx=
\mu\lim_{n\to\infty}\big(a-b\|u_n\|^2 \big)^{-1}=+\infty,
$$
and therefore $\displaystyle\lim_{n\to\infty}\|u_n\|\to+\infty$. Which is a contradiction with $\|u_n\|^2\to\frac{a}{b}$. So, \eqref{qrsl} does not hold with all $v\in H_0^1(\Omega)$
and $I$ has no $(PS)_c$ sequence with $c=\frac{a^2}{4b}$.

Case ${\displaystyle \lim_{n\to\infty}\|u_n\|^2}\neq\frac{a}{b}$,
we have $a-b\|u_n\|^2\not\to0$.
We declare that $\{u_n\}$ is bounded in $ H_0^1(\Omega)$. Otherwise, we can assume that $\|u_n\|\to+\infty$ as $n\to\infty$. In this case, we have\\[-5pt]
\begin{equation*} \displaystyle
0=\lim_{n\to\infty}\frac{c}{\|u_n\|^4}= \lim_{n\to\infty} I(u_n)
=-\frac{b}{4},
\end{equation*}
which is a contradiction by the uniqueness of limit.
So $\{u_n\}$ is bounded in $ H_0^1(\Omega)$.
That is, if necessary, pass to a subsequence (still denoted by $\{u_n\}$) and $u_0$ in $ H_0^1(\Omega)$ such that, for $n\to\infty$,\\[-6pt]
\begin{equation*}
\left\{\begin{array}{ccl}\displaystyle
u_n\rightharpoonup u_0,&\text{ weakly in }& H_0^1(\Omega),\\[4pt]
u_n\to u_0, &\text{ strongly in }& L_{\rm loc}^q (1\leq q<2^*), \\[4pt]
u_n(x)\to u_0(x), &\text{ a.e. }&x\in \Omega.
\end{array}\right.
\end{equation*}
Hence there is\\[-6pt]
\begin{equation}\label{Lebes-jian}
\begin{cases}
\langle I'(u_n), u_n\rangle= \displaystyle \bigr(a -b\|u_n\|^2\bigr)\int_{\Omega}\nabla u_n\nabla u_ndx -\mu\int_{\Omega}fu_ndx \to 0,\\
\langle I'(u_n), u_0\rangle= \displaystyle \bigr(a -b\|u_n\|^2\bigr)\int_{\Omega}\nabla u_n\nabla u_0dx -\mu\int_{\Omega}fu_0dx \to 0.
\end{cases}
\end{equation}
Lebesgue's dominated convergence theorem (see \cite[pp.27]{Walter Rudin}) leads to
\begin{align}\label{dominated convergence}
\int_{\Omega} f u_ndx = \int_{\Omega} f u_0dx + o(1).
\end{align}
From \eqref{Lebes-jian} and \eqref{dominated convergence}, we have\\[-5pt]
\begin{equation*}
\displaystyle
\big(a -b\|u_n\|^2\big)\int_{\Omega}\nabla u_n\nabla(u_n-u_0)dx \to 0.
\end{equation*}
Based on  $a-b\|u_n\|^2\not\to0$, one has $u_n\to u_0$  in $ H_0^1(\Omega)$.
This proof is complete.
\end{proof}

In order to prove main results, it is necessary to make some notes,
that is
\begin{align*}
\mathcal{D}^+ =\left\{u\in H_0^1(\Omega) \,|\, \int_{\Omega}fudx>0\right\},  ~~
\mathcal{D}^- =\left\{u\in H_0^1(\Omega) \,|\, \int_{\Omega}fudx<0\right\}.
\end{align*}
Setting $u\in \mathcal{D}^\pm$, then, there is $tu\in \mathcal{D}^\pm$ with $t>0$ and
$tu\in \mathcal{D}^\mp$ with $t<0$. For any $u\in \mathcal{D}^\pm$, it is easy to see that  $|\int_{\Omega}fudx|>0$,  and therefore, we shall let
$$\displaystyle
\Lambda:=\min_{u\in\mathcal{D}^+\bigcup\mathcal{D}^-} \left\{\|u\|^{-1}\Big|\int_{\Omega}fudx\Big|\right\}
=\min_{u\in\mathcal{D}^+} \left\{\|u\|^{-1}\int_{\Omega}fudx\right\}
\leq\frac{1}{\sqrt{S}}\|f\|_\frac{2^*}{2^*-1}.
$$


\begin{lemma} \label{lem-mountain-local-min}
Assume that $a, b>0$, $f\in L^\frac{2^*}{2^*-1}(\Omega)$ and $f(x)\geq0$ a.e.,  $u\in H_0^1(\Omega)$, then,
\begin{itemize}\rm
\item[(i)] there exist $r, \rho,\mu_{*1}>0$, such that for any $\mu\in(0,\mu_{*1}]$, it holds
$$\displaystyle
\inf_{\|u\|=r}I(u)\geq\rho ~~~\text{and}~~~ \inf_{\|u\|\leq r}I(u):=c_1<0
;
$$

\item[(ii)] There exists $R>0$, such that  $\sup_{\|u\|\geq R}I(u) \leq0$ for any $\mu>0$.
\end{itemize}
\end{lemma}
\begin{proof}[\bf Proof]
(i) Togethering with the functional \eqref{e-fun} with Sobolev imbedding inequality, we obtain
\begin{equation*} \label{eqq.1}
\left.
\begin{array}{rcl}
I(u)\geq\displaystyle \frac{a}{2}\|u\|^2 - \frac{b}{4}\|u\|^4 - \frac{\mu}{\sqrt{S}} \|f\|_{\frac{2^*}{2^*-1}}\|u\|
=\|u\| \Big(\frac{a}{2}\|u\| - \frac{b}{4}\|u\|^3 - \frac{\mu}{\sqrt{S}} \|f\|_{\frac{2^*}{2^*-1}} \Big).
\end{array}\right.
\end{equation*}
We can see that for any $\mu <\frac{a}{9b}(6abS)^\frac{1}{2}\|f\|_{\frac{2^*}{2^*-1}}^{-1}:=\mu_0$ with $r_0=(\frac{2a}{3b})^\frac{1}{2}$, there is
\begin{equation*} \label{eqq.2}
\left.
\begin{array}{rcl}\displaystyle
\mathop{I(u)}_{\|u\|=r_0, ~ \mu<\mu_0}\geq \displaystyle \Big(\frac{2a}{3b}\Big)^\frac{1}{2} \Big[\frac{a}{3} \Big(\frac{2a}{3b}\Big)^\frac{1}{2} - \frac{\mu}{\sqrt{S}}\|f\|_{\frac{2^*}{2^*-1}} \Big]:=\rho_0>0.
\end{array}\right.
\end{equation*}
Therefore, in order to calculate its conveniently, it is easy to see that there exist $r=(\frac{2a}{3b})^\frac{1}{2}$, $\rho=\frac{a^2}{9b}$,
$ \mu_{*1}=\frac{a}{18b}(6abS)^\frac{1}{2}\|f\|_{\frac{2^*}{2^*-1}}^{-1}$, such that, for any $\mu\in(0,\mu_{*1}]$, one has
\begin{align*}
\mathop{I(u)}_{\|u\|=r, ~ \mu\leq\mu_*}
&\geq \displaystyle \Big(\frac{2a}{3b}\Big)^\frac{1}{2} \Big[\frac{a}{3} \Big(\frac{2a}{3b}\Big)^\frac{1}{2} - \frac{\mu}{\sqrt{S}}\|f\|_{\frac{2^*}{2^*-1}} \Big]
\geq \displaystyle \Big(\frac{2a}{3b}\Big)^\frac{1}{2} \Big[\frac{a}{3} \Big(\frac{2a}{3b}\Big)^\frac{1}{2} - \frac{\mu_*}{\sqrt{S}}\|f\|_{\frac{2^*}{2^*-1}} \Big] \\
&\geq \displaystyle \Big(\frac{2a}{3b}\Big)^\frac{1}{2} \Big[\frac{a}{3} \Big(\frac{2a}{3b}\Big)^\frac{1}{2}
-\frac{a}{18b}(6abS)^\frac{1}{2}\|f\|_{\frac{2^*}{2^*-1}}^{-1}\cdot \frac{1}{\sqrt{S}}\|f\|_{\frac{2^*}{2^*-1}} \Big] \\
&=  \displaystyle \Big(\frac{2a}{3b}\Big)^\frac{1}{2} \Big[\frac{a}{3} \Big(\frac{2a}{3b}\Big)^\frac{1}{2}-\frac{a}{6}\Big(\frac{2a}{3b}\Big)^\frac{1}{2}\Big]
= \displaystyle \frac{a^2}{9b}=\rho.
\end{align*}

Taking $\tilde{u}\in\mathcal{D}^+\subset H_0^1(\Omega)$ with $\|\tilde{u}\|=r$, it holds $\|t \tilde{u}\|\leq r$ with $t\to0^+$, and that
\begin{align*}
\lim_{t\to0^+}\frac{I(t\tilde{u})}{t}
&=\lim_{t\to0^+}\frac{1}{t}\Big\{\frac{a}{2}\|tu\|^2-\frac{b}{4}\|tu\|^4 -\mu \int_{\Omega} f\cdot(tu)dx\Big\}\\
&=-\mu \int_{\Omega} f\tilde{u}dx=- \frac{\mu r}{\|\tilde{u}\|}\int_{\Omega} f\tilde{u}dx\\
&\leq-\mu r\Lambda=-(\frac{2a}{3b})^\frac{1}{2}\Lambda\mu <0.
\end{align*}
Therefore, $c_1:=\inf_{\|u\|\leq r}I(u)\leq-(\frac{2a}{3b})^\frac{1}{2}\Lambda\mu <0$ is well defined.

\ \ \ \ (ii)\ \ For any $\mu>0$, by Young's inequality, we have
\begin{align}\label{PS-eq.999}
I(u)
&\leq \frac{a}{2}\|u\|^2-\frac{b}{4}\|u\|^4
 + \mu\Big|\int_{\Omega}fudx\Big|  \nonumber\\
&\leq  \frac{a}{2}\|u\|^2-\frac{b}{4}\|u\|^4
 + \frac{a}{2}\|u\|^2 +\frac{\mu^2}{2aS} \|f\|_\frac{2^*}{2^*-1}^2 \nonumber\\
&=a\|u\|^2-\frac{b}{4}\|u\|^4 +\frac{\mu^2}{2aS} \|f\|_\frac{2^*}{2^*-1}^2.
\end{align}
So, there exists
$R=\frac{2}{b}\Big[a+\big(a^2+\frac{b\mu^2}{2aS} \|f\|_\frac{2^*}{2^*-1}^2\big)^\frac{1}{2}\Big]$,
such that $I(u)\leq0$ with $\|u\|\geq R$, thus the conclusion $\sup_{\|u\|\geq R}I(u) \leq0$ can be true.
This completes the proof.
\end{proof}

\subsection{Proof of the first and second solutions of problem \eqref{wt1}}

\begin{theorem} \label{thm-2-solutions}\rm
Assume that $a, b>0$ and $f(x)\in L^\frac{2^*}{2^*-1}(\Omega)$ is a positive function a.e., then, for any $\mu\in(0,\mu_*)$($\mu_*$ is defined by lemma \ref{lem-mountain-local-min}), the problem \eqref{wt1} has at least two positive solutions.
\end{theorem}

\begin{proof}[\bf Existence of the first solution]
Taking $B_{r}:=\{u\in  H_0^1(\Omega) \ |\  \|u\| \leq r\}$, where $r=(\frac{2a}{3b})^\frac{1}{2}$.
By the Lemma \ref{lem-mountain-local-min}, there exists $\mu_*>0$ such that
$\inf I({B}_{r})\leq-(\frac{2a}{3b})^\frac{1}{2}\Lambda\mu<0$ for any $\mu\in(0,\mu_*)$. By the Ekeland variational principle (see\cite[Lemma 1.1]{Ivar Ekeland}),
there is a minimizing sequence $\{u_n\}\subset\overline{{B}_{r}}$
such that
\begin{align}\label{ek1}
I(u_n)\leq \inf I({B}_{r})+\frac{1}{n}   \ \ {\rm and}\ \
I(u)\geq I(u_n)-\frac{1}{n}\|u-u_n\|
\end{align}
for all $n\in \mathbb{N}$ and for any $u\in \overline{B_r}$. Therefore, $I(u_n) \to c_1$ and $I'(u_n)\to 0$ in dual space of $ H_0^1(\Omega)$ as $n\to \infty$. By Lemma \ref{lem.PSc}, there exist a subsequence (still denoted by $ \{u_n\}$) and $u_1^*\in B_r$
such that  $u_n \to u_1^*$ as $n\to \infty$.
Then, $I(u_1^*) = c_1<0$ and $I'(u_1^*)=0$. Moreover,  $u_1^*\in \mathcal{D}^+$ and $\|u_1^*\|^2<\frac{a}{3b}$.
This implies that $u_1^*$ is a local minimizer for $I$.
Hence $u_1^*$ is a weak solution of problem \eqref{wt1}.
\end{proof}

\begin{proof}[\bf Existence of the second solution]

For any $\bar{u}\in\mathcal{D}^+$,
there exists $\theta\in(0,1)$, such that
\begin{align*}
\sup I(\bar{u})
&=\sup\Big\{\frac{a}{2}\|\bar{u}\|^2-\frac{b}{4}\|\bar{u}\|^4 -\mu \int_{\Omega} f\bar{u}dx\Big\}\\
&\leq\sup\Big\{\frac{a}{2}\|\bar{u}\|^2-\frac{b}{4}\|\bar{u}\|^4 -\mu \Lambda\|\bar{u}\|\Big\}\\
&\leq \frac{a^2}{4b} -\theta (\frac{a}{b})^\frac{1}{2}\mu \Lambda
<\frac{a^2}{4b}.
\end{align*}
According to Lemma \ref{lem-mountain-local-min}, the functional $I$ has mountain pass geometry for any $\mu\in(0,\mu_*)$.
For any $e\in  H_0^1(\Omega)$ with  $\|e\| \geq R\}$(where  $R$ defined in the lemma \ref{lem-mountain-local-min}), we set
$$
\Gamma=\Big\{\tau(t)\in C^1\big([0, 1],
 H_0^1(\Omega)\big); \tau(0)=0, \tau(1)=e\Big\}.
$$
By \eqref{e-fun}--\eqref{d-e-fun}, $I\big(\tau(t)\big)$ has continuity, $I\big(\tau(0)\big)=0$, $I\big(\tau(1)\big)\leq 0$.
So, there is
\[
\displaystyle 0<\rho \leq c_2:=\inf_{\tau\in \Gamma} \sup_{t\in[0,1]} I\big(\tau(t)\big)
\leq \sup_{u\in\mathcal{D}^+} I(u)<\frac{a^2}{4b}.
\]
By Lemma \ref{lem.PSc} and the mountain pass theorem
(see \cite[Theorem 2.1--2.4]{Ambrosetti and Rabinowitz}), there exist $u_2^*$ and a sequence $\{u_k\}$
in $ H_0^1(\Omega)$, moreover in  $\mathcal{D}^+$, such that $u_k\to u_2^*$ in
$ H_0^1(\Omega)$, $I(u_k) \to c_2=I(u_2^*)$ and
$I'(u_k) \to 0=I'(u_2^*)$ in dual space of $ H_0^1(\Omega)$.
Hence $u_2^*$ is a weak solution of problem \eqref{wt1} with
$\|u_2^*\|>(\frac{2a}{3b})^\frac{1}{2}$. Since
$I(u_1^*)<0<I(u_2^*)$, we get $u_2^*\neq u_1^*$.
\end{proof}

\begin{proof}[\bf Proof of that ${u_1^*}$ and ${u_2^*}$ are positive]
Since ${u_i^*}\in \mathcal{D}^+(i=1,2)$ are the weak solutions of  problem \eqref{wt1},
we have
$$
\Big(a-b\int_{\Omega}|\nabla u_i^*|^2dx \Big)\int_{\Omega}|\nabla u_i^*|^2dx=\mu\int_{\Omega}fu_i^*dx>0.
$$
Hence $a-b\|u_i^*\|^2>0$. This means that  $-\Delta u_i^*= \mu(a-b\|u_i^*\|^2)^{-1}f(x)\geq0$ and $u_i^*\not\equiv0$.
According to the strong maximum principle, we obtain $u_i^*$ are positive solutions.
\end{proof}

\subsection{Proof of the third solution of problem \eqref{wt1}}

\begin{lemma} \label{lem-infinity2}
Assume that $a, b,\mu>0$, $f\in L^\frac{2^*}{2^*-1}(\Omega)$  and $f(x)\geq0$ a.e., it holds that
\begin{align*}
\frac{a^2}{4b}
< \sup I(u) \leq\frac{a^2}{b} +\frac{\mu^2}{2aS} \|f\|_\frac{2^*}{2^*-1}^2.
\end{align*}
\end{lemma}
\begin{proof}[\bf Proof]
For any $u\in \mathcal{D}^+$, $tu\in \mathcal{D}^-$ with $t<0$ and we have
\begin{align}\label{PS-eq.666}
\sup_{t<0} I(tu)
&\geq\sup_{t=-(\frac{a}{b})^\frac{1}{2}\|u_0\|^{-1}}\left\{ \frac{a}{2}\|tu\|^2-\frac{b}{4}\|tu\|^4- \mu\int_{\Omega}f\cdot(tu)dx \right\}   \nonumber \\
&=\frac{a^2}{4b}+(\frac{a}{b})^\frac{1}{2}\frac{\mu}{\|u\|}\left|\int_{\Omega}fu_0dx \right|\nonumber \\
&\geq \frac{a^2}{4b}+\mu\Lambda(\frac{a}{b})^\frac{1}{2}
> \frac{a^2}{4b}.
\end{align}
Therefore, $\sup I(u)\geq\sup_{u\in \mathcal{D}^-} I(u)>\frac{a^2}{4b}$. Moreover,
via the Young's inequality we obtain that
\begin{align*}
\sup I(u)
&\leq \frac{a}{2}\|u\|^2-\frac{b}{4}\|u\|^4
 + \mu\Big|\int_{\Omega}fudx\Big|  \nonumber\\
&\leq  \frac{a}{2}\|u\|^2-\frac{b}{4}\|u\|^4
 + \frac{a}{2}\|u\|^2 +\frac{\mu^2}{2aS} \|f\|_\frac{2^*}{2^*-1}^2 \nonumber\\
&\leq \max_{t>0}\Big\{at^2-\frac{b}{4}t^4 +\frac{\mu^2}{2aS} \|f\|_\frac{2^*}{2^*-1}^2\Big\}\nonumber\\
&=\frac{a^2}{b} +\frac{\mu^2}{2aS} \|f\|_\frac{2^*}{2^*-1}^2.
\end{align*}
Consequently, $\sup I(u) \leq\frac{a^2}{b} +\frac{\mu^2}{2aS} \|f\|_\frac{2^*}{2^*-1}^2$.
This  with the \eqref{PS-eq.666}, our proof is complete.
\end{proof}

\begin{theorem} \label{thm-3th-solution}
Assume that $a, b>0$ and $f(x)\in L^\frac{2^*}{2^*-1}(\Omega)$ is a positive function a.e., then, for any $\mu>0$, the problem \eqref{wt1} has at least a negative solution.
\end{theorem}

\begin{proof}[\bf Existence of the third solution]

From Lemma \ref{lem-infinity2}, the functional $I$ has the supremum. Set
$$
\mathcal{F}=\Big\{T_t\in C^1\big( H_0^1(\Omega),  H_0^1(\Omega)\big); T_t(u)=tu, t\in\mathbb{R}\Big\},
$$
$$
\mathcal{A}=\Big\{tu; u\in\mathcal{D}^+\cup\mathcal{D}^-, t\in\mathbb{R}\Big\},
$$
Then, for all $A\in\mathcal{A}$,  $T_t(A)\in\mathcal{A}$ are hold for any $T_t\in\mathcal{F}$.
Therefore, for any $\mu>0$, there is
\[
\displaystyle \frac{a^2}{4b}<\inf_{A\in\mathcal{A}} \max_{t\in\mathbb{R}}\sup_{\tilde{u}\in\mathcal{A}} I(t\tilde{u}):=c_3
\leq \frac{a^2}{b} +\frac{\mu^2}{2aS} \|f\|_\frac{2^*}{2^*-1}^2.
\]
By Lemma \ref{lem.PSc} and applying the Minimax principle
(see \cite[Theorem 7.3.1]{Lu WD} and \cite[Theorem 1.5 \& Corollary 1.3 in Chapter 3]{Zhang GQ})
 for $I$, there exist $u_3^*$ and a sequence $\{u_m\}$
in $ H_0^1(\Omega)$, moreover in  $\mathcal{D}^-$, such that $u_m\to u_3^*$ in
$ H_0^1(\Omega)$, $I(u_m) \to c_3=I(u_3^*)$ and
$I'(u_m) \to 0=I'(u_3^*)$ in dual space of $ H_0^1(\Omega)$.
Hence $u_3^*$ is a weak solution of problem \eqref{wt1} and $\|u_3^*\|^2>\frac{a}{b}$.
\end{proof}

Indeed, instead of the Minimax principle, taking $B_{R}:=\{u\in  H_0^1(\Omega) \ |\  \|u\| \leq R\}$ and applying the Ekeland variational principle for $-I$, we can obtain the existence of $u_3^*$ by Lemma \ref{lem.PSc}, where
$R=\frac{2}{b}\big[a+ (a^2+\frac{b\mu^2}{2aS} \|f\|_\frac{2^*}{2^*-1}^2 )^\frac{1}{2}\big]$.

\begin{proof}[\bf Proof of  $u_3^*$ is negative]
Since ${u_3^*}\in \mathcal{D}^-$ is a weak solution of  problem \eqref{wt1}with  $\|u_3^*\|^2>\frac{a}{b}$,
we have $a-b\|u_3^*\|^2<0$ and
$$
\mu\int_{\Omega}f u_3^*dx=\Big(a-b\int_{\Omega}|\nabla u_3^*|^2dx \Big)\int_{\Omega}|\nabla u_3^*|^2dx<0.
$$
Hence $u_3^*\not\equiv0$ and

$$-\Delta u_3^*=\mu (a-b\|u_3^*\|^2)^{-1}f(x)\leq0.$$
By the strong minimum principle, we obtain that $u_3^*$ is negative.
This proof is complete.
\end{proof}

\subsection{Proof of the Theorem \ref{thm1.1}}

It is clear that problem \eqref{wt1} has at least two positive solutions  $u_1^*$ and $u_2^*$ for $\mu\in(0,\mu_{*1}]$  by the Theorem \ref{thm-2-solutions} and at least a negative solution  $u_3^*$ by the Theorem \ref{thm-3th-solution}.
Since
$$
I(u_1^*)<0<I(u_2^*)<\frac{a^2}{4b}<I(u_3^*),$$
 we get that $u_1^*$, $u_2^*$ and $u_3^*$ are different solutions and $0<\|u_1^*\|^2<\frac{a}{3b}<\frac{2a}{3b}<\|u_2^*\|^2<\frac{a}{b}<\|u_3^*\|^2$.
For  $\mu\in[\mu_{*1},+\infty)$, the Theorem \ref{thm-3th-solution} leads to that problem \eqref{wt1} owns at least a negative solution.
Hence there exists $\mu_*>0$ such that the problem \eqref{wt1} has at least three nontrivial solutions for $\mu\in(0,\mu_{*1})$,  and a nontrivial solution for  $\mu\in[\mu_{*1},+\infty)$.

\vspace{5pt}
Moreover, for $\|u\|^2=\frac{4a}{3b}$ and
$ \mu\leq\frac{a}{72b}(3abS)^\frac{1}{2}\|f\|_{\frac{2^*}{2^*-1}}^{-1}$,
we can obtain that $I(u)\leq\frac{a^2}{4b}$, $\sup I(u)$ is achieved on $\frac{a}{b}<\|u\|^2<\frac{4a}{3b}$.
Consequently, there exists $\mu_*=\frac{a}{72b}(3abS)^\frac{1}{2}\|f\|_{\frac{2^*}{2^*-1}}^{-1}<\mu_{*1}$,
such that, problem \eqref{wt1} has at least three nontrivial solutions  $u_1^*$, $u_2^*$  and $u_3^*$ for $\mu\in(0,\mu_*)$, and a nontrivial solution  $u_3^*$ for $\mu\geq\mu_*$. In addition, it holds that $\frac{a}{b}<\|u_3^*\|^2<\frac{4a}{3b}$ for $\mu\in(0,\mu_*)$.
This proof is completed.

\section{Proof of main results via algebraic analysis}
\textbf{Step 1}. Let $u$ be a solution, we give the calculated 形容词  method to other solutions with the help of algebraic analysis.
Since $u$ is a solution of problem \eqref{wt1}, one has
\begin{align}\label{ruojieDYy}
\big(a-b\|u\|^2 \big)\int_{\Omega}\nabla u \nabla vdx=\mu\int_{\Omega}f(x)vdx,
 ~ \forall ~ v\in  H_0^1(\Omega).
\end{align}
For given $u$, we have $\|u\|^2=\int_{\Omega}|\nabla u|^2dx:=\alpha>0$ and it is easy to see that $\alpha\neq\frac{a}{b}$. The existence of three solutions via algebraic analysis, if and only if
there exist three values of $t$ such that $tu\in H_0^1(\Omega)$ and
\begin{align}\label{ruojie2}
\big(a-b\|tu\|^2 \big)\int_{\Omega}\nabla (tu) \nabla vdx=\mu\int_{\Omega}f(x)vdx,
 ~ \forall ~ v\in  H_0^1(\Omega).
\end{align}
It follows from \eqref{ruojieDYy}--\eqref{ruojie2} that our goal is equivalent to finding all $t\in\mathbb{R}$ such that
\begin{align}\label{jiedehanshu2}
t\big(a-bt^2\alpha\big)=a-b\alpha.
\end{align}
It is easy to see that $t_0=1$ is a solution of Eq.\eqref{jiedehanshu2}, and for extra, Eq.\eqref{jiedehanshu2} has the solutions $t_1=\frac{1}{2}\Big(-1+\sqrt{\frac{4a}{b\alpha}-3}\Big)$ and $t_2=\frac{1}{2}\Big(-1-\sqrt{\frac{4a}{b\alpha}-3}\Big)$ when $\alpha<\frac{4a}{3b}$, $t_1=t_2=-\frac{1}{2}$ when $\alpha=\frac{4a}{3b}$ and no real solution when $\alpha>\frac{4a}{3b}$.
There is $t_1=t_0=1$ and $t_2=-2$ when $\alpha=\frac{a}{3b}$ in addition.

\vspace{5pt}
We set $u,\bar{u}$ are different solutions of problem \eqref{wt1} with $\mu\neq0$, then,
\begin{equation*}
\displaystyle
-\big(a-b\|u\|^2 \big)\Delta u =\mu f(x)=
-\big(a-b\|\bar{u}\|^2 \big)\Delta \bar{u}.
\end{equation*}
It holds that
 $a-b\|u\|^2\neq0$ and $a-b\|\bar{u}\|^2\neq0$ are constants by \eqref{ruojieDYy}. Consciously,
\begin{equation}\label{u=uw}
\displaystyle
-(a-b\|u\|^2)\Delta u+(a-b\|\bar{u}\|^2)\Delta\bar{u}=0,
\end{equation}
and $\big[(a-b\|u\|^2)u-(a-b\|\bar{u}\|^2)\bar{u}\big]\in H_0^1(\Omega)$
by $u,\bar{u}\in H_0^1(\Omega)$.
Multiplying the equation \eqref{u=uw} by $\big[(a-b\|u\|^2)u-(a-b\|\bar{u}\|^2)\bar{u}\big]$ and integrating over $\Omega$, we obtain that
\begin{equation*}
\Big\|(a-b\|u\|^2)u-(a-b\|\bar{u}\|^2)\bar{u}\Big\|^2=0.
\end{equation*}
Hence $u=\frac{a-b\|\bar{u}\|^2}{a-b\|u\|^2} \bar{u}$ and
all solutions of problem \eqref{wt1} are linear dependence.

\vspace{5pt}
\textbf{Step 2}. We give the proof of three solutions with  $\mu>0$ enough small only by using one of the
Mountain Pass Lemma, Ekeland variational principle, and Minimax principle.
\begin{proof}[\bf Three solutions of Theorem \ref{thm1.1}]
According to the information above,  \eqref{wt1} has at least a nontrivial solution
(denote by $u_1^*$, $u_2^*$ or $u_3^*$) with all $a,b>0$ and $\mu\in(0,\mu_*)$ by using  one of the
Mountain Pass Lemma, Ekeland variational principle, and Minimax principle.
Moreover, one has
$$
0<\|u_1^*\|^2<\frac{a}{3b}<\frac{2a}{3b}<\|u_2^*\|^2<\frac{a}{b}<\|u_3^*\|^2<\frac{4a}{3b}.
$$
Choosing one of  $u_1^*$, $u_2^*$ and $u_3^*$ denoted by $u$,
then, Eq.\eqref{jiedehanshu2} has three different solutions $t_0,t_1,t_2$ and
problem \eqref{wt1} has three different solutions $u,t_1u,t_2u$, them are linear dependence.
\end{proof}

\begin{remark}\label{zhujiM}\rm
From the step 1, we obtain that problem \eqref{wt1} has at most three different solutions.
Where we proved problem \eqref{wt1} has at least 3 solutions with $\mu>0$ enough small  (see  Theorem (1.1)),
therefore, the problem \eqref{wt1} has only three solutions with
$$
\Big\{u_1^*,u_2^*,u_3^*\Big\}=\Big\{u_1^*,t_1u_1^*,t_2u_1^*\Big\}
=\Big\{u_2^*,t_1u_2^*,t_2u_2^*\Big\}=\Big\{u_3^*,t_1u_3^*,t_2u_3^*\Big\}.
$$
\end{remark}

\textbf{Step 3}. The specific form of $\mu_{**}$ is given exactly. And the problem \eqref{wt1} has only three  solutions for $0<|\mu|<\mu_{**}$,  two solutions for $|\mu|=\mu_{**}$ and a solution for any $|\mu|>\mu_{**}$.

\begin{proof}[\bf Proof of Theorem \ref{thm1.2} with $\mu\neq0$]
Research the following elliptic problem:
\begin{equation}\label{WenTi2}
\left\{
\begin{array}{lc}\displaystyle
-\Delta u= f(x) , &\text{ in } \Omega,\\[10pt]
u=0, &\text{ in } \partial\Omega,
\end{array}
\right.
\end{equation}
where constant $a>0$, $f\in L^\frac{2^*}{2^*-1}(\Omega)$ and $f(x)\geq0$ a.e..
Then, problem \eqref{WenTi2} has an unique positive solution $U\in  H_0^1(\Omega)$.
Moreover, $\int_\Omega \nabla U\nabla vdx=\int_\Omega fvdx$ for all $v\in  H_0^1(\Omega)$ and
\begin{equation}\label{mu-GuJi}
\left\{
\begin{array}{lc}\displaystyle
\|U\|^2=\int_\Omega fUdx\leq \frac{1}{\sqrt{S}}\|f\|_\frac{2^*}{2^*-1}\|U\|\Longrightarrow
\|U\|^{-1}\geq \sqrt{S}\|f\|_\frac{2^*}{2^*-1}^{-1},\\[8pt] \displaystyle
\|U\|^2=\int_\Omega fUdx\leq \|f\|_2\|U\|_2\leq\frac{1}{\lambda_1}\|f\|_2\|U\|\Longrightarrow
\|U\|^{-1}\geq \lambda_1\|f\|_2^{-1}.
\end{array}
\right.
\end{equation}
Consider the function $g(t)$ as following:
\begin{equation}\label{FangC-t}  \displaystyle
g(t)=\big(a-b\|tU\|^2\big)t-\mu, 
\end{equation}
where $t\in \mathbb{R}$ is variable, $\mu>0$ is a parameter. If $t=T$ is a zero-point of function $g(t)$, then,
\begin{align*}
\Big(a- b \int_{\Omega}|\nabla (TU)|^2dx \Big)\int_{\Omega}\nabla (TU) \nabla vdx=\mu\int_{\Omega}fvdx, ~\forall~v\in  H_0^1(\Omega).
\end{align*}
This means that $TU$ is a solution of problem \eqref{wt1}.
Based on the $g'(t)=a-3b\|U\|^2 t^2$, we get that the zero-points of $g
(t)$ are $t_m=-\sqrt{3ab}(3b\|U\|)^{-1}$ and $t_M=\sqrt{3ab}(3b\|U\|)^{-1}$.
Therefore, $g(t)$ is monotonous decreasing along $(-\infty,t_m)$, increasing along $[t_m,t_M]$ and decreasing along $(t_M,+\infty)$,
\begin{equation}\label{gMm}
\left\{
\begin{array}{lc}\displaystyle
\min\,g(t)=g(t_m)=\Big(b\|U\|^2\cdot\frac{a}{3b\|U\|^2}-a\Big)\cdot\frac{\sqrt{3ab}}{3b\|U\|}-\mu
=-\frac{2a\sqrt{3ab}}{9b\|U\|}-\mu,\\[12pt] \displaystyle
\max\,g(t)=g(t_M)=\Big(a-b\|U\|^2\cdot\frac{a}{3b\|U\|^2}\Big)\cdot\frac{\sqrt{3ab}}{3b\|U\|}-\mu
=\frac{2a\sqrt{3ab}}{9b\|U\|}-\mu.
\end{array}
\right.
\end{equation}
Consciously, by \eqref{gMm}, it is easy to get that, there exists $\mu_{**}=2a\sqrt{3ab}(9b\|U\|)^{-1}$
such that Eq.\eqref{FangC-t} has three solutions $T_1,T_2,T_3$ for $0<|\mu|<\mu_{**}$ with
$$
T_1<-\sqrt{3ab}(3b\|U\|)^{-1}<T_2<\sqrt{3ab}(3b\|U\|)^{-1}<T_3,
$$
two solutions $T_2$ and `$T_1$ or $T_3$' for $|\mu|=\mu_{**}$ with
$$
T_1=-2\sqrt{3ab}(3b\|U\|)^{-1}<T_2=\sqrt{3ab}(3b\|U\|)^{-1}
\text{~or~}
T_2=-\sqrt{3ab}(3b\|U\|)^{-1}<T_3=2\sqrt{3ab}(3b\|U\|)^{-1},
$$
and a solution $T_1$ or $T_3$ for $|\mu|>\mu_{**}$ with
$$
T_1<-\sqrt{3ab}(3b\|U\|)^{-1}
\text{~or~}
T_3>\sqrt{3ab}(3b\|U\|)^{-1}.
$$
Besides, we can state that problem \eqref{wt1} has three solutions `$T_1U,T_2U,T_3U$' for $0<|\mu|<\mu_{**}$,
two solutions `$-2\sqrt{3ab}(3b\|U\|)^{-1}U$, $\sqrt{3ab}(3b\|U\|)^{-1}U$' for $\mu=\mu_{**}$, two solutions `$-\sqrt{3ab}(3b\|U\|)^{-1}U$, $2\sqrt{3ab}(3b\|U\|)^{-1}U$' for $\mu=-\mu_{**}$,
a solution `$T_1U$' for $\mu>\mu_{**}$ and a solution `$T_3U$' for $\mu<-\mu_{**}$,
where $U\in  H_0^1(\Omega)$ is the unique positive solution of problem \eqref{WenTi2} and  $\mu_{**}=2a\sqrt{3ab}(9b\|U\|)^{-1}$.

\vspace{5pt}

Compared with the Step 1, all solutions of problem \eqref{wt1} are linear dependence, there are no more solutions than those mentioned above.
Hence, there exists  $\mu_{**}=2a\sqrt{3ab}(9b\|U\|)^{-1}$  such that problem \eqref{wt1} has only three solutions for $0<|\mu|<\mu_{**}$, two solutions for $\mu=\pm\mu_{**}$ and a solution for $|\mu|>\mu_{**}$.
And by \eqref{mu-GuJi}, $\mu_{**}=2a\sqrt{3ab}(9b\|U\|)^{-1}\geq 2a\sqrt{3abS}(9b\|f\|_\frac{2^*}{2^*-1})^{-1}$.
\end{proof}

At last, we prove the existence of infinitely many solutions of problem \eqref{wt1} for $\mu=0$.

\begin{proof}[\bf Proof of Theorem \ref{thm1.2} with $\mu=0$]
For any  $u\in  H_0^1(\Omega)$, let $V:=\sqrt{ab}(b\|u\|)^{-1}u\in  H_0^1(\Omega)$,  it is clear that
$a-b\|V\|^2=0$ and $\int_\Omega\nabla V\nabla vdx$ are bounded with all $v\in  H_0^1(\Omega)$.
So, $V$ is a solution of problem \eqref{wt1}. According to the arbitrary of $u\in  H_0^1(\Omega)$,
we get that problem \eqref{wt1} has infinitely many solutions when $\mu=0$.
\end{proof}

We replaced the condition of $f(x)$ before by $f(x)\in L^2(\Omega)$ and $f(x)>0$ a.e., we can obtain that the conclusion in Corollary \ref{thm1.3} is clearly due to \eqref{mu-GuJi} and the progress of step 3,
where $U\in  H_0^1(\Omega)$ is the unique positive solution of problem \eqref{WenTi2} and $\mu_{**}=2a\sqrt{3ab}(9b\|U\|)^{-1}\geq 2a\lambda_1\sqrt{3ab}(9b\|f\|_2)^{-1}$.

\subsection*{Acknowledgments}
The authors would like to thank the anonymous reviewers and the editors.
We would like to thank Professor Wei Wei from Guizhou Education University too, because she has given us many suggestions for the layout of our paper.

This research was supported by the National Natural Science Foundation of China [11861021, 11661021].

\subsection*{Data Availability Statement}
Data sharing is not applicable to this article as no new data were created or analyzed in this study, and all data, methods and results have been stated before.


\end{document}